\documentclass[10pt,a4paper]{amsart}
\usepackage[english]{babel}
\usepackage{amsmath,amsthm,mathrsfs,hyperref}
\usepackage{amssymb}
\usepackage{aliascnt}
\usepackage{charter}
\usepackage[utf8]{inputenc}
\usepackage[T1]{fontenc}
\usepackage{comment}
\usepackage{xspace}
\usepackage{nicefrac}
\DeclareMathOperator{\image}{''}
\DeclareMathOperator{\acc}{acc}
\DeclareMathOperator{\dom}{dom}
\DeclareMathOperator{\range}{range}
\DeclareMathOperator{\otp}{otp}
\DeclareMathOperator{\cof}{cf}

\DeclareMathOperator{\crit}{crit}

\DeclareMathOperator{\cf}{cf}

\DeclareMathOperator{\Col}{Col}

\newcommand{\chang}{\twoheadrightarrow}

\newcommand{\Lowenheim}{L\"{o}wenheim\xspace}
\newcommand{\Todorcevic}{Todor\v{c}evi\'{c}\xspace}
\newcommand{\Godel}{G\"{o}del\xspace}
\newcommand{\Erdos}{Erd\H{o}s\xspace}
\newcommand{\ZFC}{{\rm ZFC}\xspace}
\newcommand{\GCH}{{\rm GCH}\xspace}

\newcommand{\PCF}{{\rm PCF}\xspace}

\newcommand{\xrightarrowdbl}[2][]{%
  \xrightarrow[#1]{#2}\mathrel{\mkern-14mu}\rightarrow
}

\newtheorem{theorem}{Theorem}
\newaliascnt{example}{theorem}

\aliascntresetthe{example}
\newaliascnt{fact}{theorem}

\aliascntresetthe{fact}
\newaliascnt{corollary}{theorem}
\newtheorem{corollary}[corollary]{Corollary}
\aliascntresetthe{corollary}
\newaliascnt{lemma}{theorem}
\newtheorem{lemma}[lemma]{Lemma}
\aliascntresetthe{lemma}
\newaliascnt{claim}{theorem}

\aliascntresetthe{claim}

\theoremstyle{definition}
\newaliascnt{definition}{theorem}
\newtheorem{definition}[definition]{Definition}
\aliascntresetthe{definition}
\newtheorem{question}{Question}

\newtheorem*{theorem*}{Theorem}
\newtheorem*{remark}{Remark}
\newtheorem*{example*}{Example}

\newcommand{\MM}{Q^{<\omega}}
\newcommand{\mmchang}{\xrightarrowdbl[\MM]{}}

\begin{document}

\title{Magidor-Malitz Reflection}
\author{Yair Hayut}
\begin{abstract}
In this paper we investigate the consistency and consequences of the downward \Lowenheim-Skolem-Tarski theorem for extension of the first order logic by the Magidor-Malitz quantifier.
We derive some combinatorial results and improve the known upper bound for the consistency of Chang's Conjecture at successor of singular cardinals.
\end{abstract}
\maketitle
\section{Introduction}
Let $M$ be a model of size $\lambda$ over countable language and let $\kappa$ be an infinite cardinal below $\lambda$.
The downward \Lowenheim-Skolem-Tarski theorem says that there is an elementary submodel $N\prec M$ such that $|N| = \kappa$.
This is one of the most basic results in model theory, and it can be viewed as a reflection principle. The metamathematical object which is reflected here is the first order logic. The theorem asserts that if $M$ is a model of some first order sentence and $|M| > \aleph_0$, then a \emph{strictly smaller} elementary submodel of $M$ already satisfies this sentence.

The \Lowenheim-Skolem-Tarski theorem is extremely useful in model theory, and it is quite natural that mathematicians investigate tentative generalizations of it.
One way to do this is to strengthen the underlying logic, and an important case is second order logic, $\mathcal{L}^2$. In $\mathcal{L}^2$, one can quantify over subsets and predicates of the model and thus express much more of its behavior.

Moving to second order properties, we are catapulted into the realm of large cardinals:
\begin{theorem}[Magidor]\cite{Magidor1971}
Assume that there is a cardinal $\kappa$ such that for every model $M$ with countable language, there is $N\prec_{\mathcal{L}^2} M$, $|N| < \kappa$, then there is a supercompact cardinal $\leq\kappa$.

Indeed, even full $\Pi^1_1$-reflection implies the existence of a supercompact cardinal.
\end{theorem}

Magidor's theorem demonstrates an important difference between first order and second order logic. Focusing on first order logic, there are essentially no set-theoretical restrictions on our ability to reflect valid sentences. But if we wish to reflect all the second order properties, we need a supercompact cardinal.

In this paper we try to examine what happens at some intermediate logics between first order and second order logic. We shall focus mostly on first order logic extended by the following quantifiers:
\begin{definition}[Magidor-Malitz Quantifiers]
Let $M$ be a model in the language $\mathcal{L}$. For a formula $\varphi(x_0, x_1, \dots, x_{n-1}, p_0, \dots, p_{m-1})$, we write \[M\models Q^n x_0, \dots, x_{n-1} \varphi(x_0, x_1, \dots, x_{n-1}, p_0, \dots, p_{m-1})\] if there is a set $A\subseteq M$ with $|A| = |M|$ such that \[\forall a_0, a_1, \dots, a_{n-1}\in A, \,M\models \varphi(a_0, \dots, a_{n-1}, p_0, \dots, p_{m-1}).\]

We write $M \prec_{Q^n} N$ if $M$ is an elementary submodel of $N$ with respect to first order logic enriched with the quantifier $Q^n$. We Write $M \prec_{\MM} N$ if $M\prec_{Q^n} N$ for all $n < \omega$.
\end{definition}

These quantifiers that were defined by Menachem Magidor and Jerome Malitz in \cite{MagidorMalitz1977}. In this paper, Magidor and Malitz proved that if the set theoretical principle $\diamondsuit(\omega_1)$ holds, then this quantifier satisfies a certain compactness theorem for models of size $\aleph_1$. This was generalized by Shelah in \cite{Shelah-MM-compactness-lambda} to arbitrary successor of regular cardinal, $\lambda^+$, under the assumption $\diamondsuit(\lambda) + \diamondsuit(\lambda^+)$. In \cite{ShelahRubin}, Shelah and Rubin showed that the the quantifiers $Q^n$ form a strict hierarchy and discussed some cases in which compactness fails.

Basically, the Magidor-Malitz quantifiers express \emph{some} of the second order properties of the model. This weakening enables us to get consistently some variants of the \Lowenheim-Skolem-Tarski theorem at accessible cardinals. See for example Theorem~\ref{thm: mm reflection from huge} and section \ref{subsec: mm to aleph1 successor to regular}. The hierarchy of languages between first order and second order logic reflects at the size of the large cardinals which are needed in order to get reflection principles for those logics. For example, one can get $Q^1$-reflection at successor cardinal, starting from subcompact cardinal, but in order to get a similar reflection principle relative to $Q^{<\omega}$ using the same approach, one needs $\Pi^1_1$-subcompact cardinal.

A more concrete approach is related to the celebrated Chang's conjecture. Recall:
\begin{definition}
Let $\kappa, \lambda, \mu, \nu$ be cardinals, $\kappa > \lambda$, $\mu > \nu$. We say that $(\kappa, \lambda)\xrightarrowdbl{} (\mu, \nu)$ is for every model $M$ of the countable language $\mathcal{L}$ with distinct unary predicate $A$ such that $|M| = \kappa$ and $|A|=\lambda$ there is an elementary submodel $N\prec M$ such that $|N| = \mu, |A\cap N| = \nu$.
\end{definition}
Chang's conjecture is a natural strengthening of the model theoretical two cardinals theorems of Vaught and Chang. There is an extensive literature about Chang's conjecture for various parameters. See, for example, \cite[section 7.3]{ChangKeisler1990}.

It turned out that some instances of Chang's conjecture are equivalent to reflection of sentences with the Chang's quantifier $Q^1$ (see below, Lemma~\ref{lem: cc for q1 ref}). Some results in this paper point towards a similarity between Chang's conjecture and the reflection of $Q^{<\omega}$, while other demonstrate the difference between them. For example:
\begin{theorem*}
It is consistent, relative to a $(+\omega+1)$-subcompact cardinal, that \[(\aleph_{\omega+1}, \aleph_{\omega})\chang (\aleph_1, \aleph_0).\]
\end{theorem*}
But in Question~\ref{ques: mm reflection at aleph omega} we ask whether the same situation can occur at all for $Q^{<\omega}$.

Our notation is standard. We work in \ZFC. We force downwards (namely, $p \leq q \implies p  \Vdash q\in \dot{G}$, where $\dot{G}$ is the canonical name for the generic filter). For a formula $\varphi(x_0, \dots, x_{n-1}, p)$ with free variables $x_0, \dots, x_{n-1}$ and parameter $p$ (for simplicity, we assume that there is only one parameter), a set $A\subseteq M$ is called \emph{$\varphi$-cube} if for all $a_0, \dots, a_{n-1}\in A$, $M\models \varphi(a_0, \dots, a_{n-1}, p)$.

The paper is arranged in three sections. In section \ref{section: the MM quantifier} we define the $\MM$ analogue for Chang's conjecture and derive some reflection principles from it. In section \ref{section: consistency results} we investigate the large cardinals which imply $\MM$ reflection and prove consistency results about some cases of $\MM$ reflection at small cardinals.
\section{Combinatorial Consiqueces of \texorpdfstring{$\MM$}{MM} reflection}\label{section: the MM quantifier}
In this section we analyze the relationship between the reflection of the Magidor-Malitz quantifiers, $\MM$, and some square like principles.

The Magidor-Malitz quantifiers allow us to access some of the second order properties of the model. As we will see, the downward \Lowenheim-Skolem-Tarski theorem for the quantifiers $Q^n$ is a strong reflection principle, yet it consistently holds for some pairs of small cardinals (assuming the consistency of large cardinals).

\begin{definition}
Let $\lambda,\mu$ and $\eta$ be cardinals. $\lambda\xrightarrowdbl[\MM]{\eta} \mu$ iff for every model of cardinality $\lambda$, over a language of cardinality $\eta$, there is a $\MM-$elementary submodel of cardinality $\mu$. When $\eta=\aleph_0$ we write $\lambda\mmchang \mu$. Similarly, $\lambda\xrightarrowdbl[Q^n]{} \mu$ iff for every model of cardinality $\lambda$ there is a $Q^n-$elementary submodel of cardinality $\mu$.

$\lambda\xrightarrowdbl[\MM]{} <\mu$ abbreviates the assertion that for every model of cardinality $\lambda$ there is a $\MM-$elementary submodel of cardinality less than $\mu$.
\end{definition}

Let us recall that a weak instance of the reflection principle $\lambda\mmchang\mu$ is equivalent to Chang's conjecture:
\begin{lemma}[Folklore]\label{lem: cc for q1 ref}
Let $\mu < \lambda$ be cardinals. $\lambda^+\xrightarrowdbl[Q^1]{} \mu^+$ iff $(\lambda^+, \lambda)\chang (\mu^+, \mu)$.
\end{lemma}
\begin{proof}
Let us assume that $\lambda^+\xrightarrowdbl[Q^1]{} \mu^+$. Let $(M, A)$ be a model of type $(\lambda^+, \lambda)$.

Assume, without loss of generality, that:
\begin{enumerate}
\item $A$ is a predicate in the language
\item There is a definable well ordering $\leq^\star$ on $M$ with order type $\lambda^+$
\item For every $a\in M$ there is a definable surjection from $A$ onto the elements that are smaller than $a$ in $\leq^\star$.
\end{enumerate}
Let $N \prec_{Q^1} M$ be an elementary submodel of cardinality $\mu^+$. We claim that $A^N = A \cap N$ has cardinality $\mu$.

$M\models \neg Q^1 x\in A$ (since $|A| = \lambda < |M| = \lambda^+$) and therefore $N\models \neg Q^1 x\in A$, so $|A^N|\leq \mu$. On the other hand, by elementarity for every $a\in N$ there is a surjection from $A^N$ onto $\{b\in N\mid b \leq^\star a\}$ so $|A|^N$ cannot be strictly smaller than $\mu$.

On the other hand, assume that $(\lambda^+, \lambda)\chang (\mu^+, \mu)$. By enriching the language, we may assume that for every formula $\phi(x, b)$ there is a function symbol $f_\phi$ such that 
\[\{f_\phi(x, b)\mid x\in M\} = \{y\in M\mid M\models \phi(y, b).\}\] 
Moreover, if the set of $x\in M$ such that $\phi(x, b)$ has cardinality $\leq\lambda$ then we pick $f_\phi$ such that:
\[\{f_\phi(x, b)\mid x\in A\} = \{y\in M\mid M\models \phi(y, b).\}\]
and if there are $\lambda^{+}$ many elements $x\in M$ such that $\phi(x, b)$ we pick $f_\phi$ to be one to one.

Let $N \prec M$ be an elementary submodel with $|A^N|=\mu$. Let us look at the formula $Q^1 x \phi(x, b)$. If it holds in $M$, then $f_{\phi}(x, b)$ enumerates the set of witnesses and when restricting this function to $N$, we get a one to one function from $N$ such that $N\models \forall x \phi(f_\phi(x, b), b)$. Thus $|\{x \in N \mid \phi(x,b)\}| = |N|$. On the other hand, if $\neg Q^1 x \phi(x,b)$ then \[M\models \forall x \phi(x,b)\rightarrow \exists a\in A,\, x = f_\phi(a, b).\]
Therefore, $N$ satisfies the same formula and $|\{x \in N \mid \phi(x,b)\}| \leq |A^N|$.
\end{proof}
The previous lemma shows that the reflection principle $\mu^+\mmchang \kappa$ is at least as strong as Chang's Conjecture. For example, since $(\aleph_{\omega + 1}, \aleph_{\omega})\not\chang (\aleph_n, \aleph_{n-1})$ for all $n\geq 4$, we conclude that $\aleph_{\omega + 1}\not\mmchang \aleph_n$ for all $n \geq 4$.

The proof of the lemma shows that if $\lambda^+\xrightarrowdbl[Q^1]{} \mu$ then $\mu$ must be a successor cardinal and in particular regular. Similarly, if $\lambda^+\xrightarrowdbl[Q^1]{}\ < \kappa$ then we may assume always that the cardinalities of the elementary submodels are successor cardinals.

Let us start with the following useful observation which shows that models that are obtained from Chang's conjecture can be assumed to have a specific order type.
\begin{lemma}\label{lem: ordertype is regular}
Assume $\lambda\xrightarrowdbl[Q^1]{} \mu$. Then for every model $\mathcal{A}$ on set of ordinals of order type $\lambda$ there is an elementary submodel, $\mathcal{B}$, such that $\otp \mathcal{B} = \mu$.
\end{lemma}
\begin{proof}
Assume that the language of $\mathcal{A}$ has a predicate $<$, interpreted as the order of the ordinals. Let us reflect the statement:
\[\forall \alpha \neg Q^1\beta,\,\beta < \alpha\]
from $\mathcal{A}$ into $\mathcal{B}$. Observe that for every $\alpha\in\mathcal{B}$, $\otp (\mathcal{B}\cap \alpha) < \mu$. Therefore $\mathcal{B}$ is an increasing union of chain of models, ordered by end-extension, where the order type of the chain is $\mu$ and all the substructures in the chain has order type strictly smaller than $\mu$.
\end{proof}

Let us recall the definition of $\square(\kappa)$. Let $A$ be a set of ordinals, we denote: \[\acc A = \{\beta \mid \sup A\cap \beta = \beta\}.\]
\begin{definition}\label{def: square of kappa}
Let $\mathcal{C} = \langle C_\alpha \mid \alpha < \kappa\rangle$ be a sequence of closed sets such that:
\begin{enumerate}
\item $\sup C_\alpha = \alpha$ for all limit ordinal $\alpha$.
\item If $\beta\in \acc C_\alpha$ then $C_\alpha \cap \beta = C_\beta$.
\item There is no club $D$ such that $\forall \alpha \in (\acc D) \cap \kappa$, $D\cap \alpha = C_\alpha$.
\end{enumerate}
Then $\mathcal{C}$ is called a $\square(\kappa)$ sequence. We say that $\square(\kappa)$ holds if there is a $\square(\kappa)$ sequence.
\end{definition}
This definition, due to \Todorcevic, is pivotal in the research of reflection properties, in particular when dealing with $\Pi^1_1$-statements. See \cite{Rinot2014} for extensive review.
\begin{theorem} Assume that $\kappa\xrightarrowdbl[Q^2]{} \mu$ where $\mu$ is regular. Then $\square(\kappa)$ fails.
\end{theorem}
\begin{proof}
Let $\mathcal{C} = \langle C_\alpha \mid \alpha < \kappa\rangle$ be a coherent $C$-sequence, i.e., a sequence that satisfies the first two conditions in Definition~\ref{def: square of kappa}. We will show that there is a \emph{thread}, namely a club $D$ such that for every $\alpha \in (\acc D) \cap \kappa$, $D\cap \alpha = C_\alpha$.

Let $\mathcal{A}\prec H(\chi)$, for some large enough regular $\chi$, with $\kappa + 1 \subseteq \mathcal{A}$, $\mathcal{C}\in \mathcal{A}$ and $|\mathcal{A}| = \kappa$.

Let $\mathcal{B}\prec_{Q^2} \mathcal{A}$ with $|\mathcal{B}| = \mu$ and assume that $\kappa, \mathcal{C} \in \mathcal{B}$. Since $\mu$ is a regular cardinal, by Lemma~\ref{lem: ordertype is regular}, we can take $\mathcal{B}$ so that $\sup (\mathcal{B}\cap \kappa) = \rho$, $\cf \rho = \mu$.

Let us look at $\delta\in \acc C_\rho \cap \acc (\mathcal{B}\cap \kappa)$ below $\rho$. Let $\beta = \min (\mathcal{B}\cap \kappa \setminus \delta)$. $\beta$ is well defined, since $\delta < \rho = \sup (\mathcal{B} \cap \kappa)$. Let us show that $\delta \in \acc C_\beta$. If $\delta = \beta$, then this is clearly true, so let us assume that $\delta \neq \beta$.

Let $\alpha \in \mathcal{B}\cap \beta$. By the minimality of $\beta$, $\alpha < \delta$. Let $\gamma$ be $\min C_\beta \setminus \alpha$. This ordinal is definable from $\alpha, \beta, \mathcal{C}$ and therefore $\gamma \in \mathcal{B}$. Since $C_\beta$ is cofinal at $\beta$, $\gamma < \beta$. By using the minimality of $\beta$ again, we conclude that $\gamma < \delta$. Therefore, $\delta$ is an accumulation point of ordinals in $\mathcal{B} \cap C_\beta$ and in particular $\delta \in \acc C_\beta$.

Since $\mathcal{C}$ is coherent, we conclude that $C_\delta = C_\beta \cap \delta$, i.e.\ $C_\beta$ is an end extension of $C_\delta$ which we denote by $C_\delta \trianglelefteq C_\beta$.

Now, let $\delta < \delta^\prime$ be in $\acc C_\rho \cap \acc (\mathcal{B}\cap \kappa)$. Let $\beta = \min (\mathcal{B}\cap \kappa \setminus \delta)$, $\beta^\prime = \min (\mathcal{B}\cap \kappa \setminus \delta^\prime)$. We claim that $C_\beta \trianglelefteq C_{\beta^{\prime}}$, since otherwise there is some $\gamma < \beta$ such that $\gamma \in C_\beta \triangle C_{\beta^\prime}$. Such $\gamma$ must appear in $\mathcal{B}$ (by elementarity), so it is smaller than $\delta$. But $C_\delta \trianglelefteq C_{\delta^\prime} \trianglelefteq C_{\beta^\prime}$ - a contradiction.

We conclude that
\[\mathcal{B}\models Q^2 \alpha, \beta < \kappa,\ \alpha \geq \beta \bigvee\ C_\alpha \trianglelefteq C_\beta\]
Therefore, $\mathcal{A}$ contains a set of cardinality $\kappa$, $I$, of elements which are compatible in $\mathcal{C}$. $D = \bigcup_{\alpha \in I} C_\alpha$ is a thread.
\end{proof}
\subsection{The tree property at successor of singular}
In this section we will show that reflection of the $Q^2$-quantifier can behave, in some ways, similarly to the existence of strongly compact cardinals. In particular, we will show that in the successor of singular limit of cardinals in which some $\MM$-reflection holds, the tree property holds.

Let us recall the following definition:
\begin{definition}\cite{MagidorShelah96}
A triplet $\mathcal{S} = \langle I, \kappa, \mathcal{R}\rangle$ is called a \emph{system} if:
\begin{enumerate}
\item $\mathcal{R}$ is a set of partial orders on $I \times \kappa$.
\item For every $\alpha < \beta$ in $I$ there are $\zeta < \xi < \kappa$ and $\leq_i \in \mathcal{R}$ such that $\langle \alpha, \xi\rangle \leq_i \langle \beta, \zeta\rangle$.
\item For every $\leq_i\in \mathcal{R}$, $\langle \alpha, \xi\rangle \leq_i \langle \beta, \zeta\rangle$ implies that either $\alpha < \beta$ or that $\alpha = \beta$ and $\xi = \zeta$.
\item For every $\leq_i\in \mathcal{R}$, if $\alpha \leq \beta$, $\langle \alpha, \xi\rangle \leq_i \langle \gamma, \rho\rangle$ and $\langle \beta, \zeta\rangle \leq_i \langle \gamma, \rho\rangle$ then $\langle \alpha, \xi\rangle \leq_i \langle \beta, \zeta\rangle$.
\end{enumerate}

The system $\mathcal{S}$ is \emph{narrow} if $\kappa^+, |\mathcal{R}|^+ < \sup I$.

A branch in the system $\mathcal{S}$ is a partial function $b\subseteq I\times \kappa$ such that there is $\leq_i\in\mathcal{R}$ such that for every $\alpha < \beta$ in $\dom b$, $\langle \alpha, b(\alpha)\rangle \leq_i \langle \beta, b(\beta)\rangle$. A branch $b$ is \emph{cofinal} if $\sup \dom b = \sup I$.

The \emph{Narrow System Property} holds at a cardinal $\nu$ if every narrow system $\mathcal{S} = \langle I, \kappa, \mathcal{R}\rangle$ with $\sup I = \nu$ has a cofinal branch.
\end{definition}
For full discussion about narrow systems and the Narrow System Property, see \cite{Hanson2015}. Systems and Narrow Systems appear naturally when dealing with the tree property at successor of singular cardinals. Those narrow systems are usually restrictions of a given tree (which is assumed to be a partial order on $\nu^{+} \times \nu$, partial to the lexicographic order) to some rectangle $I \times \kappa$ in a way that still preserve a significant portion of the properties of the original tree. 

\begin{theorem}\label{thm: tree property from mm reflection}
Let $\mu$ be a singular cardinal and assume that $\langle\kappa_i\mid i < \cf \mu\rangle$ is cofinal in $\mu$, $\kappa_0 \geq \cf \mu$, $\cf \kappa_i = \kappa_i$ for all $i$. If for every $i < \cf \mu$, $\mu^+\xrightarrowdbl[Q^2]{\kappa_i}\kappa_{i+1}$ then the tree property holds at $\mu^+$.
\end{theorem}
\begin{proof}
We prove the theorem in two steps. First we apply $Q^2$-reflection in order to find for a given $\mu^+$-tree $T$ a narrow subsystem. At this step we will use $\mu^+\xrightarrowdbl[Q^2]{\cf \mu} \lambda$ for some regular $\lambda < \mu^+$. Then we pick $i$ large enough so that $\kappa_i$ is larger than the width of the system and use $\mu^+\xrightarrowdbl[Q^2]{\kappa_i} \kappa_{i+1}$ in order to get a branch through the narrow system.

Let $T$ be a $\mu^+$-tree and assume, without loss of generality, that $T = \langle \mu^+\times\mu, \leq_T\rangle$, i.e.\ that the $\alpha$-th level of $T$ is given by $\{\alpha\}\times\mu$. Let $\mathcal{A}_0$ be an elementary substructure of $H(\chi)$ for some large enough $\chi$, such that $|\mathcal{A}_0|=\mu^+$, $\mu^+\subseteq \mathcal{A}_0$ and $T\in\mathcal{A}_0$.

Let $\mathcal{B}_0$ be a $Q^2$-elementary substructure of $\mathcal{A}_0$, containing $\cf \mu$ such that the order type of $\mathcal{B}_0\cap \mu^+$ has cofinality above $\cf \mu$ and $|\mathcal{B}_0| < \mu$. We may assume, without loss of generality, that $\{\kappa_i\mid i< \cf \mu\} \subseteq \mathcal{B}_0$. Let $\Delta = \mathcal{B}_0\cap \mu^+$ - the set of levels that appear in $\mathcal{B}_0$. Let $\delta = \sup \Delta$ and let us consider the branch below $\langle\delta, 0\rangle$. Since $T$ is a tree, for every $\alpha\in \Delta$ there is $\zeta < \mu$ such that $\langle\alpha,\zeta\rangle\leq_T \langle\delta,0\rangle$. Since $\mu$ is singular and $\cf \mu\subseteq \mathcal{B}_0$, there is some $i\in \mathcal{B}_0$ such that $\zeta < \kappa_i$. 

If $\alpha, \beta$ are both in $\Delta$, and $\langle\alpha,\zeta\rangle, \langle\beta,\xi\rangle\leq_T \langle\delta,0\rangle$ then $\langle\alpha,\zeta\rangle \leq_T \langle\beta,\xi\rangle$. If we assume that $\zeta,\xi < \kappa_i$ then by elementarity there are $\tilde\zeta, \tilde\xi \in \mathcal{B}_0\cap \kappa_i$ such that $\langle\alpha,\tilde\zeta\rangle \leq_T \langle\beta,\tilde\xi\rangle$.

Since the cofinality of $\delta$ is larger than $\cf \mu$, there is $i < \cf \mu$ such that $\mathcal{B}_0$ satisfies the $\MM$-formula:
\[Q^2 \alpha, \beta \exists \zeta,\xi < \kappa_i,\, \langle\alpha,\zeta\rangle \leq_T \langle\beta,\xi\rangle\]

By $\MM$-elementarity there is some subset $I\subseteq\mathcal{A}_0$ with cardinality $\mu^+$ such that every element of $I$ is an ordinal and the elements of $I$ satisfy the same compatibility relation. Therefore, we can define a narrow system on $I\times \kappa_i$ (with only single relation), by the restriction of the tree $T$ to this set.

Let us show that the Narrow System Property follows from the reflection assumption $\mu^+\xrightarrowdbl[Q^2]{\kappa_i}\kappa_{i+1}$ for cofinal set of regular $\kappa_i<\mu$.

\begin{lemma}\label{claim: NSP from CCMM}
Let $\mu$ be a singular cardinal and assume that for cofinal set of regular cardinals $\kappa < \mu$, there is a regular cardinal $\lambda$, such that $\kappa < \lambda < \mu$ and $\mu^+\xrightarrowdbl[\MM]{\kappa}\lambda$. Then the narrow system property holds at $\mu^+$.
\end{lemma}
\begin{proof}
Let $\mathcal{S} = \langle I, \kappa, \mathcal{R}\rangle$ be a narrow system with height $\mu^+$, $|\mathcal{R}| \leq \kappa$.

Let $\mathcal{A}_1$ be an elementary substructure of $H(\chi)$ containing all ordinals in $\mu^+$ and $\mathcal{S}$. Let us pick a $\MM-$elementary substructure $\mathcal{B}_1$ of $\mathcal{A}_1$, of cardinality strictly larger than $\kappa$, containing all ordinals below $\kappa$. Let $\delta = \sup (\mathcal{B}_1 \cap \mu^+)$ and let us pick some element $\epsilon\in I\setminus\delta$. Since $\mathcal{S}$ is a narrow system, for every $\alpha\in \mathcal{B}_1$ there are $\zeta, \xi < \kappa$ and index $i < \kappa$ such that $\langle\alpha,\zeta\rangle\leq_i\langle\epsilon,\xi\rangle$. By \autoref{lem: ordertype is regular}, we may assume that $\otp \mathcal{B}_1$ is regular and therefore, for unbounded many ordinals below $\epsilon$ in $\mathcal{B}_1$ the tuple $(\zeta,\xi, i)$ is constant. Therefore, for some $\zeta_\star, i_\star < \kappa$, $\mathcal{B}_1$ satisfies:

\[Q^2 \alpha, \beta,\, \langle\alpha,\zeta_\star\rangle \leq_{i_\star} \langle\beta,\zeta_\star\rangle\]

The same holds in $\mathcal{A}_1$, and therefore there is a branch in $\mathcal{S}$.
\end{proof}

Applying Lemma \ref{claim: NSP from CCMM} on $T\restriction I$, we obtain a cofinal branch through $T\restriction I$, $b^\prime$. The set $\{s\in T\mid \exists s\in b,\,t\leq s\}$ is a cofinal branch through $T$.

\end{proof}

The assumptions of Theorem~\ref{thm: tree property from mm reflection} and Lemma~\ref{claim: NSP from CCMM} can be weakened to the assumption that for every model $\mathcal{A}$ of cardinality $\mu^+$ over a language of cardinality $\eta < \mu$ there is some regular cardinal $\kappa < \mu$ and a $Q^2$-elementary submodel $\mathcal{B}$ of cardinality $\kappa$ (note that in this case, $\eta < \kappa$). This is true, since the proof does not use the fact that the values of the cardinals $\lambda_n$ are pre-determined. The proof only uses the fact that for every $\eta < \mu$ (which is the width of the narrow system), we can find a $Q^2$-elementary submodel of some fragment of the universe, $\mathcal{B}$, such that $\eta \subseteq \mathcal{B}$ and $\otp \mathcal{B}$ is regular and large enough.

It is interesting to compare Theorem \ref{thm: tree property from mm reflection} to the theorem of Shelah and Magidor:
\begin{theorem}\cite{MagidorShelah96}
Let $\nu$ be a singular limit of cardinals which are $\nu^+$-strongly compact. Then the tree property holds at $\nu^+$.
\end{theorem}
The reflection principle which is required for Theorem \ref{thm: tree property from mm reflection} follows from large cardinals at the level of partial supercompact (see Theorem~\ref{thm: MM elementariness from pi11-subcompact}). It is unclear whether one can derive this kind of reflection from strongly compact cardinals. In fact, it is unclear even if one can derive some instances of Chang's Conjecture from strongly compact cardinals. On the other hand, the reflection principle \[\lambda \mmchang \kappa\] itself does not imply that $\lambda$ or $\kappa$ are large cardinals (see \ref{subsec: mm to large successor of regular}).

The assumption that $\mu^+\mmchang\kappa$ for cofinally many $\kappa < \mu$ seems to be stronger than the narrow system property. For example, it cannot hold for $\mu = \aleph_{\omega}$, since $(\aleph_{\omega + 1}, \aleph_{\omega})\not\chang(\aleph_{n+1}, \aleph_n)$ for every $n \geq 3$. This fact is a combination of $\PCF$ related results of Cummings, Foreman, Magidor and Shelah. For a proof, see \cite[Section 4]{SharonViale2010}.
\section{Consistency results}\label{section: consistency results}
This section is dedicated to the derivation of some consistency results regarding the reflection principles that were defined above.
\subsection{Chang's conjecture at \texorpdfstring{$\aleph_{\omega+1}$}{alphaomega}}\label{subsec: cc at aleph omega}
We begin this section with two theorems about the consistency of Chang's conjecture at successor of singular cardinals.
\begin{definition}[Jensen]\cite{NeemanSteelSubcompact}
A cardinal $\kappa$ is \emph{$(+\alpha)$-subcompact} if for every $A\subseteq H(\kappa^{+\alpha})$ there are $\rho < \kappa$, $B\subseteq H(\rho^{+\alpha})$ and an elementary embedding 
\[j\colon \langle H(\rho^{+\alpha}), \in, \rho, B\rangle \to \langle H(\kappa^{+\alpha}), \in, \kappa, A\rangle\]
where $\rho$ is the critical point of $j$. A cardinal $\kappa$ is \emph{subcompact} if it is $(+1)$-subcompact.
\end{definition}

In order to get a general feeling about the place of this type of cardinals in the large cardinal hierarchy, let us remark that if $\kappa$ is $\kappa^{+\omega + 1}$-supercompact and $\kappa^{+\omega}$ is strong limit then $\kappa$ is $(+\omega + 1)$-subcompact and it has a normal measure concentrating on the set of cardinals below it which are $(+\omega + 1)$-subcompact. On the other hand, if a cardinal $\kappa$ is $(+\omega+1)$-subcompact then it is $\kappa^{+n}$-supercompact for every $n < \omega$ and it has a normal measure concentrating on cardinals $\rho$ which are $\rho^{+n}$ supercompact of all $n < \omega$.

\begin{lemma}[Folklore]
Let $\kappa$ be $(+\alpha)$-subcompact cardinal, where $\kappa^{+\alpha}$ is regular and $|H(\kappa^{+\alpha})| = \kappa^{+\alpha}$. Then there is a generic extension in which $\square(\kappa^{+\alpha})$ holds and $\kappa$ is still $(+\alpha)$-subcompact. 
\end{lemma}
\begin{proof}
Let $\lambda = \kappa^{+\alpha}$. Let $\mathbb{P}$ be the forcing notion for adding a square sequence for $\kappa^{+\alpha}$ using bounded approximations. Namely, the conditions of $\mathbb{P}$ are sequences of the form $\langle C_\alpha \mid \alpha \leq \delta < \lambda\rangle$ where $C_\alpha \subseteq \alpha$ is a club at $\alpha$ and if $\beta \in \acc C_\alpha$ then $C_\alpha \cap \beta = C_\beta$. We order $\mathbb{P}$ by end-extension, namely $p \leq q$ if $p$ end extends $q$. 

It is well known that $\mathbb{P}$ is $\lambda$-strategically closed and that if $G \subseteq \mathbb{P}$ is a generic filter then $\bigcup G$ is a $\square(\lambda)$ sequence in $V[G]$. Moreover, by the distributivity of $\mathbb{P}$, $(H(\lambda))^{V[G]} = H(\lambda)^V$ Let us claim that $\kappa$ is still $(+\alpha)$-subcompact in the generic extension. 

Assume otherwise. Let $\dot{A}$ be a name for a subset of $H(\lambda)$, and let $p$ be a condition that forces that there is no $\rho < \kappa$ and $B \subseteq H(\rho^{+\alpha})$ such that there is an elementary embedding 
\[j\colon \langle H(\rho^{+\alpha}), B, \in\rangle \to \langle H(\lambda), \dot{A}, \in\rangle.\]
Since $j\in H(\lambda)^{V[G]}$, $j$ is a member of the ground model $V$. Therefore, one can enumerate all candidates for $j$ in a sequence of length $\lambda$, $\langle j_\xi \mid \xi < \lambda\rangle$. Let us also enumerate the elements of $H(\lambda)$ in a sequence of length $\lambda$, $\langle a_\xi \mid \xi < \lambda\rangle$.

As $\mathbb{P}$ is strategically closed, we can construct a sequence of conditions $p_\xi \in \mathbb{P}$ and sets $\langle M_\xi \mid \xi < \lambda\rangle$ such that for $\zeta < \xi$, $p_\xi \leq p_\zeta$, $p_0 \leq p$ and
\begin{enumerate}
\item $p_\xi \Vdash M_\xi \prec \langle H(\lambda), \dot{A}, \in\rangle$
\item\label{covering of H(lambda)} $a_\xi \in M_\xi$. For all $\rho < \xi$, $M_\rho \subseteq M_\xi$.
\item $p_\xi$ decides for all $x \in M_\xi$ whether $x\in \dot{A}$ or not.
\item The range of $j_\xi$ is contained in $M_\xi$.
\item $j_\xi$ not an elementary embedding from $\dom j_\xi$ to $M_\xi$. 
\end{enumerate}
Let \[\tilde{A} = \{x\in H(\lambda) \mid \exists \xi < \lambda,\, p_\xi \Vdash x\in \dot{A}\}.\]
By condition (\ref{covering of H(lambda)}), $\bigcup_{\xi < \lambda} M_\xi = H(\lambda)$. Thus, by applying Tarski-Vaught's test we conclude that for all $\xi$, \[\langle M_\xi, M_\xi \cap \tilde{A}, \in\rangle \prec \langle H(\lambda), \tilde{A}, \in\rangle\]. 

In $V$, $\kappa$ is $(+\alpha)$-subcompact and therefore there is some $\rho < \kappa$, $B$ and \[j\colon \langle H(\rho^{+\alpha}), B, \in\rangle \to \langle H(\lambda), \tilde{A}, \in\rangle\] 
elementary. For some $\xi < \lambda$, $j = j_\xi$. Therefore, $p_\xi$ forces that $j$ is not elementary as a map to $M_\xi$. But $M_\xi$ is an elementary submodel of $\langle H(\lambda), \tilde{A}, \in\rangle$ and contains the image of $j$ - a contradiction.  
\end{proof}

Before stating the main theorems of this section, let us recall the following characterization of Chang's Conjecture.
\begin{lemma} The following are equivalent:
\begin{enumerate}
\item $(\kappa, \lambda) \chang (\mu, \nu)$
\item For every function $f\colon \kappa^{{<}\omega} \to \lambda$ there is $I \subseteq \kappa$, $|I| = \mu$ such that \[|f\image I^{{<}\omega}| \leq \nu.\]
\item For every function $f\colon \kappa^{{<}\omega} \to \kappa$ there is $I \subseteq \kappa$, $|I| = \mu$ such that \[|\big(f\image I^{{<}\omega}\big)\cap \lambda| \leq \nu.\]
\end{enumerate}
\end{lemma}
The proof is done by using Skolem functions for one direction and by adding $f$ to the model in the other direction.
\begin{theorem}\label{thm: cc subcompact}
Let $\kappa$ be $(+\omega+1)$-subcompact cardinal, $\kappa^{+\omega}$ strong limit. There is $\rho <\kappa$ such that $(\kappa^{+\omega+1},\kappa^{+\omega})\chang(\rho^{+\omega+1},\rho^{+\omega})$.\end{theorem}
\begin{remark}
This is an improvement of the current upper bound for this type of Chang's Conjecture, which is slightly above huge (namely, a cardinal $\kappa$ such that there is an elementary embedding $j\colon V\to M$, $\crit j = \kappa$ and $M$ is closed under sequences of length $j(\kappa)^{+\omega + 1}$), see \cite{LevinskiMagidorShelah}.
\end{remark}
\begin{proof}
Let $\mu=\kappa^{+\omega}$. Assume otherwise, and let us pick for every $\rho < \kappa$ a function $f_{\rho}\colon(\mu^{+})^{<\omega}\to\mu^{+}$ such that for all $R\subseteq\mu^{+}$ of cardinality $\rho^{+\omega+1}$, $|f_{\rho}^{\prime\prime}[R]^{<\omega}\cap\mu|\neq\rho^{+\omega}$.

Let us code this sequence of functions as a subset of $H(\kappa^{+\omega+1})$, $A$.

Let $j\colon \langle H(\rho^{+\omega+1}), \in, B\rangle \to \langle H(\kappa^{+\omega +1}), \in, A\rangle$ be an elementary embedding as in the definition of $(+\omega+1)$-subcompactness. $B$ codes a sequence of functions from $(\rho^{+\omega+1})^{<\omega}$ to $\rho^{+\omega+1}$, $\langle g_\eta \mid \eta < \rho\rangle$, witnessing the failure of Chang's conjecture. Note that $\rho^{+\omega}$ is strong limit.

Let us look at $f_{\rho}$. Let $R=j^{\prime\prime}\rho^{+\omega + 1}\in H(\kappa^{+\omega+1})$. By our assumption, $|f_{\rho}^{\prime\prime}[R]^{<\omega}\cap \mu|>\rho^{+\omega}$. Let $n$ be the first ordinal such that $|f^{\prime\prime}[R]^{<\omega}\cap \kappa^{+n}|=\rho^{+\omega+1}$.

Since $\cof\rho^{+\omega}=\omega$ and $\rho^{+\omega+1}$ is regular, it is impossible that $n=\omega$, so $n$ is a natural number.

Let $\langle\vec{\alpha}_{\xi}\mid \xi<\rho^{+\omega+1}\rangle$ be a sequence of elements in $(\rho^{+\omega+1})^{<\omega}$, and assume that $\langle f_{\rho}(j(\vec{\alpha_\xi}))\mid \xi<\rho^{+\omega+1}\rangle$ is strictly increasing sequence of ordinals below $\kappa^{+n}$.

By elementarity, for every $\xi\neq \xi^{\prime}$,
\[\langle g_{\eta}(\vec{\alpha}_{\xi})\cap \rho^{+n}\mid\eta<\rho\rangle\neq\langle g_{\eta}(\vec{\alpha}_{\xi^{\prime}})\cap\rho^{+n}\mid\eta<\rho\rangle\]
Otherwise, for every $\tilde{\rho}<\kappa$ we would get that
\[f_{\tilde{\rho}}(j(\vec{\alpha}_{\xi}))\cap\kappa^{+n}=f_{\tilde{\rho}}(j(\vec{\alpha}_{\xi^{\prime}}))\cap\kappa^{+n}\]
and evaluating at $\tilde{\rho}=\rho$ we get a contradiction.

The number of possible sequences of this form is $(\rho^{+n})^{\rho}$ which is strictly smaller than $\rho^{+\omega}$ - a contradiction.
\end{proof}
\begin{theorem}\label{thm: cc aleph omega from subcompact}
Let $\kappa$ be a $(+\omega+1)$-subcompact cardinal, $\kappa^{+\omega}$ strong limit.
There is $\rho < \kappa$ such that forcing with $\Col(\omega,\rho^{+\omega})\times\Col(\rho^{+\omega+2},<\kappa)$ forces the instance of Chang's conjecture $(\aleph_{\omega+1},\aleph_{\omega})\chang(\aleph_{1},\aleph_{0})$.
\end{theorem}
\begin{proof}
Let $\mu = \kappa^{+\omega}$.

Assume otherwise, and let us pick for all $\rho<\kappa$ a $\Col(\omega,\rho^{+\omega})\times\Col(\rho^{+\omega+2},<\kappa)$-name for a function $f_{\rho}\colon(\mu^{+})^{<\omega}\to\mu^{+}$, witnessing the failure of Chang's conjecture in the generic extension. Note that we can still code this set of names as a subset of $H(\kappa^{+\omega+1})$, $A$.

Let $j\colon \langle H(\rho^{+\omega+1}),\in ,B\rangle\to \langle H(\kappa^{+\omega+1}), \in, A\rangle$ be the subcompact embedding. As in the previous proof, we denote by $\dot{g}_\eta$ the names of the functions coded by $B$. Let us look at $\dot{f}_\rho$ and the forcing $\Col(\omega,\rho^{+\omega})\times\Col(\rho^{+\omega+2},<\kappa)$.

Let $R=j^{\prime\prime}\rho^{+\omega+1}$.

By the assumption, $\Vdash|f_{\rho}^{\prime\prime}[R]^{<\omega}\cap \mu| = \rho^{+\omega+1}$ and by the same argument as before, there is some condition $(p_{0},p_{1})$ and a minimal ordinal $n$ such that
\[(p_{0},p_{1})\Vdash|f_{\rho}^{\prime\prime}[R]^{<\omega}\cap \kappa^{+n}| = \rho^{+\omega+1}.\]
Since $\rho^{+\omega+1}$ is still regular after the forcing, $n$ must be finite.

Let $\{\dot{a_\xi}\mid \xi < \rho^{+\omega+1}\}$ be a sequence of names of finite sequences of ordinals below $\rho^{+\omega+1}$ such that it is forced by the empty condition that \[f_\rho(j(\dot a_\xi)) < f_\rho(j(\dot a_{\xi^\prime})) < \kappa^{+n}\] for all $\xi < \xi^\prime$.

Since the $\Col(\rho^{+\omega+2}, < \kappa)$ is $\rho^{+\omega+2}$-closed, we can find a condition that below it the value of $\dot a_\xi$ is determined only by the first coordinate for all $\xi < \rho^{+\omega+1}$.

There are $\rho^{+\omega}$ many conditions in $\Col(\omega,\rho^{+\omega})$. Therefore, there is a single condition $p$ and a set of size $\rho^{+\omega+1}$ of finite sequences such that $p$ decides the value of all of them. Let $\{b_{\xi}\mid \xi<\rho^{+\omega+1}\}$ be an enumeration of this set.

Back in $H(\rho^{+\omega+1})$, for every pair of ordinals $\xi<\xi^{\prime}$ there is $\eta < \rho$ and a condition $q\in \Col(\omega,\eta^{+\omega}) \times \Col(\eta^{+\omega + 2}, <\rho)$ that forces $g_{\eta}(b_{\xi})< g_{\eta}(b_{\xi^{\prime}})<\rho^{+n}$.

This defines a coloring $[\rho^{+\omega+1}]^{2}\to \rho \times V_{\rho}$. Let us restrict the coloring to the first $(2^{\rho^{+n}})^{+}$ elements. By the \Erdos-Rado theorem, there is a homogeneous set of cardinality $\rho^{+n + 1}$, $H$. Let $(\eta, q)$ be the color of all pairs in $H$. So for every $\xi < \xi^\prime$ in $H$,
\[q \Vdash_{\Col(\omega, \eta^{+\omega})\times\Col(\eta^{+\omega+2}, <\rho)} \dot{g}_\eta(b_\xi) < \dot{g}_\eta(b_{\xi^\prime}) < \rho^{+n}\]
and in particular, after forcing with $\Col(\omega, \eta^{+\omega})\times\Col(\eta^{+\omega+2}, <\rho)$ below the condition $q$, there is a set of order type $\rho^{+n+1}$ below $\rho^{+n}$, which is impossible.
\end{proof}

Assuming the consistency of a $(+\omega+1)$-subcompact cardinal, it is consistent that $\kappa$ is $(+\omega+1)$-subcompact and $\square(\kappa^{+\omega+1})$ holds (yet $\square^\star_{\kappa^{+\omega}}$ and even the approachability property must fail, by Theorem \ref{thm: cc subcompact}). Therefore in the model of Theorem~\ref{thm: cc aleph omega from subcompact} we may have $\square(\aleph_{\omega+1})$ and therefore $\aleph_{\omega+1}\not\xrightarrowdbl[\MM]{} \aleph_1$.

The assumption in both Theorem~\ref{thm: cc subcompact} and Theorem~\ref{thm: cc aleph omega from subcompact} is slightly below the assumption of $\kappa$ being $\kappa^{+\omega+1}$-supercompact.

\begin{question}
Is Chang's conjecture $(\aleph_{\omega+1},\aleph_\omega)\chang(\aleph_1,\aleph_0)$ consistent assuming the consistency of a strongly compact cardinal?
\end{question}

\subsection{MM submodels}
In this subsection we investigate which large cardinal assumptions imply the $\MM$-reflection. We first deal with the case $\lambda\mmchang\kappa$ for $\lambda$ successor cardinal.
\subsubsection{\texorpdfstring{$\Pi^1_1$}{pi11} Subcompact cardinals}
The following large cardinal notion was defined by Neeman and Steel in \cite{NeemanSteelSubcompact}. We will use a slightly different notation than the one used in \cite{NeemanSteelSubcompact}.
\begin{definition}
A cardinal $\kappa$ is $\Pi^1_1$-$(+\alpha)$-subcompact, if for every $A\subseteq H(\kappa^{+\alpha})$ and $\Pi^1_1$-statement $\phi$ such that $\langle H(\kappa^{+\alpha}), \kappa, \in, A\rangle\models \phi$ there is $\rho < \kappa$ and $B\subseteq H(\rho^{+\alpha})$ such that $\langle H(\rho^{+\alpha}), \in, \rho, B\rangle\models \phi$.\footnote{In \cite{NeemanSteelSubcompact}, Neeman and Steel denoted by $\Pi^2_1$-subcompact the large cardinal notion that is denoted here by $\Pi^1_1$-$(+1)$-subcompact.}
\end{definition}

In order to get a general feeling about the place of $\Pi^1_1$-subcompact cardinals in the large cardinal hierarchy, we remark that a $\Pi^1_1$-($+0$)-subcompact is weakly compact, while $(+0)$-subcompact cardinal is inaccessible cardinal\footnote{This depends on the precise definition of $H(\chi)$ for singular $\chi$. If we define $H(\chi)$ to be the collection of sets of cardinality $<\chi$ such that every member of them belongs to $H(\chi)$ we conclude that $(+0)$-subcompact cardinal is Mahlo.}.

\begin{lemma}
Let $\kappa$ be a $\Pi^1_1$-$(+\alpha)$-subcompact cardinal, $\alpha < \kappa$. Then there is a stationary subset of $\kappa$ of $(+\alpha)$-subcompact cardinals.
\end{lemma}
\begin{proof}
Note that the notion of $(+\alpha)$-subcompact cardinal is not changed when one weakens the assumption of $j$ to assuming only that $j$ is $\Sigma_1$ elementary relative to additional predicate (by coding the full elementary diagram). Using this interpretation, the statement "$\kappa$ is $(+\alpha)$-subcompact" is $\Pi^1_1$-statement and therefore reflects downwards. By adding a predicate $C$ for a club, we obtain the desired result.
\end{proof}

\begin{theorem}\label{thm: MM elementariness from pi11-subcompact}
Let $\kappa$ be $\Pi^1_1$-$(+\alpha)$ subcompact, $\alpha < \kappa$ successor ordinal and assume that $|H(\kappa^{+\alpha})| = \kappa^{+\alpha}$. Then $\kappa^{+\alpha}\mmchang < \kappa$.
\end{theorem}
\begin{proof}
Let $\lambda = \kappa^{+\alpha}$.

Let $\mathcal{A}$ be an algebra on $\lambda$. $\mathcal{A}$ can be coded by a single predicate on $H(\lambda)$, $A$. Moreover, we assume that $A$ codes also the truth predicate of $\mathcal{A}$ and that the language of $\mathcal{A}$ contains some bijection between $H(\lambda)$ and $\lambda$.

For every formula $\varphi$ in the language of the model $\mathcal{A}$ with the $\MM$ quantifiers we enrich the language of $\mathcal{A}$ by adding one function symbol. For formula $\varphi$ of the form $Q^n x_0, \dots, x_{n-1} \psi(x_0, \dots, x_{n-1}, b)$, let us add the function symbol $F_\varphi\colon \mathcal{A}\times\mathcal{A}\to\mathcal{A}$ and interpret it such that whenever $\mathcal{A}\models \varphi(b)$ (where $b\in\mathcal{A}$), the function $x\to F_\varphi(x, b)$ is one to one and its image, $I$, witnesses $\varphi$. Namely, \[\forall x_0, \dots, x_{n-1}\in \mathcal{A},\ \mathcal{A}\models \psi(F_\varphi(x_0, b),\dots,F_\varphi(x_{n-1}, b), b).\] We will assume that the truth predicate $A$ contains also the truth value of all formulas in the enriched language.

We want to code the fact that $A$ is a truth predicate for $\MM$-formulas into a single $\Pi^1_1$-sentence.  

Let $\Phi$ be the following $\Pi^1_1$-sentence: For every $X\subseteq H(\lambda)$ one of the following cases hold:
\begin{enumerate}
\item There is $y\in X$ which is not of the form $\langle \phi, p, x\rangle$ where $\phi$ is a (\Godel number of a) $\MM$-formula, $x,p\in H(\lambda)$.
\item There is a pair of elements $\langle \phi, p, x\rangle, \langle \phi^\prime, p^\prime, x^\prime\rangle\in X$ with $\langle \phi, p\rangle \neq \langle \phi^\prime, p^\prime\rangle$
\item $\phi = Q^n x_0, \dots, x_{n-1} \varphi(x_0, \dots, x_{n-1}, p)$ and there are $a_0, \dots, a_{n-1}$ such that $\forall i < n,\,\langle \phi, p, a_i\rangle\in X$ and $\mathcal{A}\models\neg\varphi(a_0, \dots, a_{n-1})$.
\item $X$ is bounded.
\item ``$\phi(p)$'' belongs to the truth predicate of $\mathcal{A}$.
\end{enumerate}
Where all the truth values are evaluated using the truth predicate.

Clearly, the formula $\Phi$ holds in $\mathcal{A}$. 

Let $\rho$, $B$ and $j\colon \langle H(\rho^{+\alpha}),\in, B\rangle\to \langle H(\kappa^{+\alpha}), \in, A\rangle$ witness the assumption that $\kappa$ is $\Pi^1_1$-$(+\alpha)$-subcompact relative to the $\Pi^1_1$ formula $\Phi$.

Let us claim that $\mathcal{B} = j^{\prime\prime} H(\rho^{+\alpha})$ is a $\MM$ elementary substructure of $\mathcal{A}$.

We need to show that for a $\MM$-formula $\varphi$, and $b\in \mathcal{B}$, $\mathcal{B}\models \varphi(b)$ if and only if $\mathcal{A}\models \varphi(b)$.

We prove the claim by induction on the complexity of $\varphi$. Elementarity for first order quantifiers and connectives follows from the elementarity of $j$. Let us assume that $\varphi$ has the form $Q^n x_0, \dots, x_{n-1} \psi(x_0, \dots, x_{n-1}, y)$, and that the induction hypothesis holds for all formulas in the complexity level of $\psi$.

If $\mathcal{A}\models \varphi(b)$, then $g(x)=F_\varphi(x, b)$ enumerates some set $I$ such that for every $a_0,\dots, a_{n-1}\in I$, $\mathcal{A}\models \psi(a_0,\dots, a_{n-1})$. By elementarity of $j$, when restricting $g$ to elements of $\mathcal{B}$ its range will be a subset of $\mathcal{B}$ which is a $\psi$-block of cardinality $|\mathcal{B}|$.

Let us assume that $\mathcal{B}\models\varphi$. Recall that $j$ is an isomorphism between $H(\rho^{+\alpha})$ and $\mathcal{B}$. Thus, 
\[\langle H(\rho^{+\alpha}), \in, B\rangle\models \exists I,\, |I|  \text{ unbouneded } \forall a_0,\dots, a_{n-1}\in I,\, \psi(a_0,\dots, a_{n-1}, j^{-1}(b))\]
where $I$ is the preimage under $j$ of the subset of $\mathcal{B}$ which witnesses $\varphi$. Let $X$ to be $\{\langle \varphi, b, x\rangle \mid x \in I\}$. So $\langle H(\rho^{+\alpha}), \in, B\rangle$ is a model for $\neg \Phi$ as witnessed by $X$, unless $\mathcal{A}\models \varphi(b)$.
\end{proof}
Theorem~\ref{thm: MM elementariness from pi11-subcompact} is parallel to Theorem~\ref{thm: cc subcompact}. Unfortunately, we do not know how to generalize the stronger result of Theorem~\ref{thm: cc aleph omega from subcompact}. For successors of regulars and target $\aleph_1$, subsection~\ref{subsec: mm to aleph1 successor to regular} gives some partial results.
\subsubsection{Inaccessible cardinals} For inaccessible cardinals, the consistency strength seems to be lower.
\begin{theorem}
Let $\kappa$ be Ramsey cardinal. Then for every regular cardinal $\omega < \mu<\kappa$, $\kappa\mmchang \mu$.
\end{theorem}
\begin{proof}
Let $\mathcal{A}$ be an algebra on $\kappa$. Let $I$ be a set of indiscernibles for $\mathcal{A}$ and let $\mathcal{B}$ be the substructure of $\mathcal{A}$ generated by the first $\mu$ indiscernibles.

As in the previous proof, in order to show that $\mathcal{B}\prec_{\MM}\mathcal{A}$, we may enrich the language of $\mathcal{A}$ by functions that produce witnesses for all $\MM$-formulas that hold in $\mathcal{A}$ and show only that for every $\MM$-formula $\varphi = Q^n x_0,\dots,x_n \psi$, if $\mathcal{B}\models\varphi$ then $\mathcal{A}\models\varphi$.

Let $J\subseteq \mathcal{B}$ be any set of cardinality $\mu$. Every element $a\in J$ can be represented as $f(\alpha_0, \dots, \alpha_{m-1})$ where $f$ is one of the Skolem functions of $\mathcal{A}$ and $\alpha_i\in I$.

Since there are only countably many Skolem functions, $f$, there is some fixed $f_\star$ and uncountable subset of $J$, $K$, such that for every $a\in K$ there are indiscernibles $\alpha_0, \dots, \alpha_{m-1}$ such that $a = f_\star(\alpha_0,\dots, \alpha_{m-1})$. Moreover, if $\gamma$ is the maximal indiscernible that appears in the description of the parameters of the formula $\psi$, we may assume that for all $a\in K$, $a=f_\star(\alpha_0, \dots, \alpha_{m-1})$ and the set $\{\alpha_i \mid i < m,\,\alpha_i\leq\gamma\}$ is fixed (since $\mu$ is a regular cardinal, and there are less than $\mu$ finite sequences of indiscernibles below $\gamma$).

By $\Delta$-system arguments, there is some finite set $r\in I^k$ and a set $\tilde{J}\subseteq I^{m-k}$ such that $|\tilde{J}|=\mu$ and $f_\star(r^\smallfrown s)\in I$ for all $s\in \tilde{J}$. Moreover, we may assume that for all $s\neq s^\prime$, $f_\star(r^\smallfrown s)\neq f_\star(r^\smallfrown s^\prime)$ and $\max s < \min s^\prime$ or $\min s > \max s^\prime$. Otherwise, by indiscerniblity, every two members of $K$ were equal.

Since the members of $I$ are indiscernible, and by our assumption on $K$, we have that for every $\beta_0 < \beta_1 < \cdots < \beta_{(m - k) \cdot n - 1}$ in $I$ if we let \[b_i = f_\star(r^\smallfrown \langle \beta_{(m-k)i}, \beta_{(m-k)i + 1}, \dots, \beta_{(m-k)(i+1) -1}\rangle)\] then $\psi(b_0, \dots, b_{n-1})$. This provides a set of cardinality $\kappa$ in $\mathcal{A}$ which is a $\psi$-cube.
\end{proof}

\subsection{\texorpdfstring{$\MM$}{MM}-reflection at successor cardinals}\label{subsec: mm to large successor of regular}
In this subsection we will discuss some cases in which the one can force the $\MM$-reflection at successor of regular cardinals, starting from large cardinals at the level of huge cardinals.

We will start with the following technical definition:
\begin{definition}
Let $\mathbb{P}$ be a forcing notion. We say that $\mathbb{P}$ is $\kappa$-$\MM$ preserving if for every algebra of cardinality $\kappa$, $\mathcal{A}$, $\MM$-formula $\varphi$, and $G\subseteq\mathbb{P}$ a generic filter
\[V\models \mathcal{A}\models\varphi\iff V[G]\models\mathcal{A}\models\varphi.\]
\end{definition}

The class of $\kappa$-$\MM$ preserving forcings is closed under finite iterations. Note that in order to show that a forcing is $\kappa$-$\MM$-preserving, it is enough to show that for every formula of the right signature $\varphi$, if $V[G]\models \mathcal{A}\models Q^n \varphi$ then also $V\models \mathcal{A}\models Q^n \varphi$.
\begin{lemma}\label{lem: kappa-closed is MM preserving}
$\kappa$-closed forcing notion is $\kappa$-$\MM$ preserving.
\end{lemma}
\begin{proof}
Let $\mathbb{P}$ be a $\kappa$-closed forcing notion and let $\mathcal{A}$ be a model of cardinality $\kappa$. Let $\varphi$ be a formula and assume that \[V[G] \models \mathcal{A} \models Q^n x_0, \dots, x_{n-1}\ \varphi(x_0, \dots, x_{n-1}, \vec{p}).\]
By induction on the complexity of the formula $\varphi$ we may assume that the satisfaction of $\varphi$ is absolute between $V$ and $V[G]$. Let $\dot{I}$ be a $\mathbb{P}$-name for a subset of $\mathcal{A}$, and let $p_0\in \mathbb{P}$ be a condition that forces $|\dot{I}| = \kappa$ and that $I$ is a $\varphi$-block. Let us construct, by induction, a decreasing sequence of conditions $p_i \in \mathbb{P}$, of length $\kappa$, such that $p_{i + 1} \Vdash \check{a}_i \in \dot{I}$ for some $a_i \in \mathcal{A}$, and for every $i\neq j$, $a_i \neq a_j$.

By the $\kappa$-closure - this is possible. Let $J = \{a_i \mid i < \kappa\}$. $J$ is a $\varphi$-block since for every $\vec{a} \in J^n$, if we take $\xi$ to be above all indices of the elements of $\vec{a}$, \[p_\xi\Vdash \mathcal{A} \models \varphi(\vec{a}, \vec{p})\]
and therefore $\mathcal{A} \models \varphi(\vec{a}, \vec{p})$. Thus, in $V$, $J$ is a $\varphi$ block of size $\kappa$.
\end{proof}

A forcing $\mathbb{P}$ has the $(\lambda, \kappa, <\zeta)$-c.c.\ if every set of cardinality $\lambda$ of conditions, $A$, has a subset $B\subseteq A$ of cardinality $\kappa$ such that for every $C \subseteq B$, $|C| < \zeta$, there is a lower bound for $C$. We will be interested in the case of $(\kappa, \kappa, <\omega)$-c.c.\ which is a minor strengthening of $\kappa$-Knaster property. We use the following terminology: For a forcing notion $\mathbb{P}$, we say that $\mathbb{P}$ has precaliber-$\kappa$ if it is $(\kappa, \kappa, <\omega)$-c.c.

\begin{lemma}
Assume that $\kappa$ is regular. Every forcing notion that has precaliber-$\kappa$ is $\kappa$-$\MM$ preserving.
\end{lemma}
\begin{proof}
Let $\mathbb{P}, \mathcal{A}, \dot{I}$ and $\varphi$ be as in the proof of \ref{lem: kappa-closed is MM preserving}. Let us construct a sequence of conditions $p_i\in \mathbb{P}$ such that $p_i \Vdash \check{a}_i \in \dot{I}$ and for every $i\neq j$, $a_i \neq a_j$. Note that we cannot assume that $p_i$ is compatible with $p_j$ for every $i, j$.

Since $\mathbb{P}$ has precaliber-$\kappa$, there is a subset $J \subseteq \kappa$ such that for every finitely many elements from $J$, $\xi_0, \dots, \xi_{m-1}$, the conditions $p_{\xi_0}, \dots, p_{xi_{m-1}}$ have a common lower bound.

In particular, for every $n$ elements from $J$, $\xi_0, \dots, \xi_{n-1}$, there is a condition $q\in \mathbb{P}$ stronger than $p_{\xi_0}, \dots, p_{xi_{n-1}}$ and $q$ forces $\varphi(\vec{a}, \vec{p})$ where $\vec{a} = \langle a_{\xi_0}, \dots, a_{\xi_{n-1}}\rangle$. Therefore, $\{a_\xi \mid \xi \in J\}$ is a $\varphi$-block in the ground model.
\end{proof}
We remark that $\kappa$-c.c.\ forcing notions may not be $\kappa$-$\MM$ preserving (e.g.\ a forcing that adds a branch to a $\kappa$-Suslin tree does not preserve the $\MM$ sentence "there is no set of cardinality $\kappa$ of incompatible elements").

\begin{lemma}
If there is a projection from $\mathbb{P}$ onto $\mathbb{Q}$ and $\mathbb{P}$ is $\kappa$-$\MM$ preserving then $\mathbb{Q}$ is also $\kappa$-$\MM$ preserving.
\end{lemma}

A $\kappa$-$\MM$ preserving forcing notion does not collapse $\kappa$. Otherwise, if $|\kappa| = \mu$ in the generic extension, for some $\mu < \kappa$, then the truth value of the $\MM$ formula $Q^1 x,\, x < \mu$ in the model $\langle \kappa, \leq\rangle$ was changed.

For the next theorem we would like to have a forcing notion that collapses cardinals below a Mahlo cardinal and behaves nicely under iterations and elementary embeddings. There are several such forcing notions in the literature (see \cite{Kunen1978}, \cite{ForemanLaver1988}, \cite{eskew2014measurability}, \cite{Shioya2011} and others).

For our results we will use a simple variation of the forcing notion that was defined in \cite{Shioya2011}. The arguments for the properties of this forcing are mostly due to Shioya. 

We would like to thank Eskew for pointing our an error in the previous version of this definition. 

\begin{definition}
Let $\mu < \kappa$ be regular cardinals. Let $\mathbb{S}(\mu, <\kappa)$ be the Silver collapse between $\mu$ and $\kappa$. Namely, the $\mu^{+}$ support product of $\Col(\mu, \alpha)$ for every $\alpha \in [\mu, \kappa)$.  

We denote by $\mathbb{EC}(\mu, <\kappa)$ the Easton support product $\prod_{\mu \leq \alpha < \kappa} \mathbb{S}(\alpha, <\kappa)$, where the product ranges over regular cardinals.

Namely, $\mathbb{EC}(\mu, <\kappa)$ is the set of all partial functions such that:
\begin{enumerate}
\item $\dom(f)\subseteq \{\langle\alpha, \beta, \gamma\rangle\mid \mu \leq \alpha < \kappa, \beta\in[\alpha, \kappa) \text{ regular cardinals, }\gamma < \alpha\}$.
\item $f(\alpha,\beta,\gamma)\in \beta$.
\item $|\{\alpha \mid \exists \langle \alpha,\beta,\gamma\rangle\in \dom(f)\}\cap \rho| < \rho$ for all inaccessible $\rho$.
\item For all $\alpha < \kappa$, $|\{\langle \beta, \gamma\rangle\mid \langle\alpha,\beta,\gamma\rangle\in\dom(f)\}| \leq \alpha$.
\item For all $\alpha < \kappa$, $\beta < \kappa$, $|\{\gamma\mid \langle\alpha,\beta,\gamma\rangle\in\dom(f)\}| < \alpha$.
\end{enumerate}

\end{definition}
\begin{lemma}\label{lem: Knaster collapse}
Let $\mu < \kappa$ be regular cardinals and assume that $\kappa$ is Mahlo. Let $\mathbb{P} = \mathbb{EC}(\mu, <\kappa)$.
\begin{enumerate}
\item $\mathbb{P}$ has precaliber-$\kappa$ and it is $\mu$-closed.
\item $\mathbb{P}$ collapses every cardinal between $\mu$ and $\kappa$.
\end{enumerate}
\end{lemma}
\begin{proof}
Let $\{p_i \mid i < \kappa\}$ be a sequence of conditions in $\mathbb{P}$. Let \[g(i) = \sup \{\alpha, \beta, \gamma, p_i(\alpha, \beta, \gamma)\mid \langle\alpha,\beta,\gamma\rangle\in\dom(p_i)\}.\]
For all $i < \kappa$, $g(i) < \kappa$ and therefore there is a club $C\subseteq\kappa$ of cardinals such that $\forall\rho\in C, \sup g^{\prime\prime}(\rho) \leq \rho$. Let $\{\rho_i \mid i < \kappa\}$ be an increasing enumeration of all the strongly inaccessible cardinals in $C$. Note that this is a stationary subset of $\kappa$.

Let $q_i = p_i\restriction [\rho_i, \kappa)\times [\rho_i, \kappa)$ namely the function $p_i$ restricted to inputs of the form $\langle\alpha,\beta,\gamma\rangle$ where $\rho_i \leq \alpha,\beta$. By the definition of $C$, for every $i < j$, $q_i$ and $q_j$ are compatible, since their domain are disjoint - the domain of $q_i$ is a subset of $\rho_{i+1}\times\rho_{i+1}\times\rho_{i+1}$ while the domain of $q_i$ does not contain any triplet of the form $\langle\alpha,\beta,\gamma\rangle$ with $\alpha, \beta < \rho_j$. Similarly, for every finite collection of element $q_{i_0}, \dots, q_{i_{n-1}}$, the union $\bigcup_k q_{i_k}$ is a condition, stronger than all $q_{i_k}$.

Let us look at $r_i = p_i\restriction \rho_i\times \rho_i$. Since $\rho_i$ is strongly inaccessible, $r_i\in V_{\rho_i}$ (its domain is bounded below $\rho_i$). By fixing some enumeration of $V_\kappa$ that maps elements of $V_\rho$ to ordinals below $\rho$ for every inaccessible $\rho$, the function $\rho_i\to r_i$ is equivalent to a regressive function on a stationary set. Therefore, by Fodor's lemma, there is a stationary subset $S$, such that for all $\rho_i, \rho_j\in S$, $r_i = r_j$. We conclude that for every $\rho_i, \rho_j\in S$, $p_i$ is compatible with $p_j$, and furthermore - for every finite collection $p_{i_0}, \dots ,p_{i_{n-1}}$, such that $i_k \in S$, the union $\bigcup p_{i_k}$ is a condition stronger than each $p_{i_k}$.
\end{proof}

\begin{theorem}\label{thm: mm reflection from huge}
Let $\mu < \kappa \leq \lambda < \delta$ be regular cardinals and assume that there is an elementary embedding $j\colon V\to M$ with $j(\kappa) = \delta$, $M^{j(\lambda)} \subseteq M$.

Let $\mathbb{P} = \mathbb{EC}(\mu, <\kappa)$ and let $\dot{\mathbb{Q}}$ be the $\mathbb{P}$-name for the forcing notion $\mathbb{EC}(\lambda, <\delta)$ as defined in $V^{\mathbb{P}}$.

Then $V^{\mathbb{P} \ast \dot{\mathbb{Q}}}\models j(\lambda)\mmchang \lambda$.
\end{theorem}
\begin{proof}
By Lemma \ref{lem: Knaster collapse}, $\mathbb{P}$ has precaliber-$\kappa$ and $\dot{\mathbb{Q}}$ is forced to have precaliber-$\delta$ in $V^{\mathbb{P}}$.

Let $\mathbb{N}$ be the termspace forcing for $\dot{\mathbb{Q}}$ and let $\mathbb{R}$ be $\mathbb{EC}(\lambda ,< \delta)^V$.
\begin{lemma} There is a projection from $\mathbb{R}$ onto $\mathbb{N}$.
\end{lemma}
\begin{proof}
Using the fact that $\kappa$ is a huge cardinal and in particular $\delta$-supercompact, for every regular cardinal, $\rho \geq \kappa$, $\rho^{<\kappa} = \rho$. In particular, we can construct a bijection between all the nice $\mathbb{P}$-names of ordinals below a regular $\rho$ and $\rho$ (using the fact that $\mathbb{P}$ is $\kappa$-c.c.). Using this bijection we can identify a condition in $\mathbb{R}$ as a partial function to names of ordinals in the appropriate domain and get a projection. 

This projection is continuous in the following sense: if $A \subseteq \mathbb{R}$ is a collection of conditions and $A$ has a lower bound then (by the properties of $\mathbb{R}$) it has an unique greatest lower bound, $\bigcup A$, and the projection sends $\bigcup A$ to the unique lower bound of the image of $A$.
\end{proof}

By elementarity, $j(\mathbb{P} \ast \dot{\mathbb{Q}}) = j(\mathbb{P})\ast j(\dot{\mathbb{Q}})$. We want to show that one can find a weak master condition for this forcing.

$j(\mathbb{P}) = \prod_{\mu\leq \alpha < \delta}\mathbb{S}(\alpha, < \delta)$, with Easton support. Let us decompose $j(\mathbb{P})$ in the following way: 

\[j(\mathbb{P}) = \big(\prod_{\mu \leq \alpha < \kappa} \mathbb{S}(\alpha, < \delta)\big) \times \big(\prod_{\kappa \leq \alpha < \lambda} \mathbb{S}(\alpha, < \delta)\big) \times \big(\prod_{\lambda \leq \alpha < \delta} \mathbb{S}(\alpha, < \delta)\big)\]
were all products are with Easton support. The first component projects onto $\mathbb{P}$, by taking the projection of each component $\mathbb{S}(\alpha, <\delta)$ and restrict it to its first $\kappa$ coordinates. The last coordinate is $\mathbb{R}$.

Thus, there is a projection from $j(\mathbb{P})$ to $\mathbb{P} \times \mathbb{R}$ and in particular to $\mathbb{P} \ast \dot{\mathbb{Q}}$. This projection respects greatest lower bounds. Therefore, after forcing with $j(\mathbb{P})$ we have a generic filter $G\subseteq \mathbb{P}$ and a generic filter $H \subseteq \dot{\mathbb{Q}}$. The set $\tilde{q} = \bigcup_{q\in H} j(q)$ is a condition: its domain is bounded by $\sup j\image \delta < j(\delta)$, and therefore it is Easton. For every inaccessible $\alpha < \delta$, and every $\beta < j(\alpha)$, $\tilde{q}(\beta)$ is the union of at most $\alpha$ many functions, and thus it is a condition in the relevant Silver collapse. Finally, the support of $\tilde{q}$ in the product $\mathbb{S}(\rho, <j(\delta))$ is at most $\delta \cdot \rho \leq \rho$ for all $\rho \geq \delta$\footnote{This is the only place in which using Silver collapse (instead of the more standard Levy collapse) is important. In the original version of this paper, the Levy collapse forcing was used and this argument was flawed. We would like to thank Eskew for pointing out this mistake.}.  

Therefore, by using the directed closure of the forcing, $\tilde{q}$ is a condition.

Using Silver criteria for extending elementary embeddings to generic extension, one can extend $j$ by forcing with $j(\mathbb{P})\ast j(\dot{\mathbb{Q}})/(\mathbb{P} \ast \dot{\mathbb{Q}})$. $j(\dot{\mathbb{Q}})$ is a $j(\mathbb{P})$-name for a highly directed closed (at least $\delta$-closed). We claim that  $j(\mathbb{P})\ast j(\dot{\mathbb{Q}})/(\mathbb{P} \ast \dot{\mathbb{Q}})$ is $j(\lambda)$-$\MM$-preserving.

By the structure of the projection - we can split the discussion into two parts. First, note that $j(\dot{\mathbb{Q}})$ is forced to be $j(\lambda)$-closed in $V^{j(\mathbb{P})}$, and therefore $j(\lambda)$-$\MM$-preserving. $j(\mathbb{P})/ (\mathbb{P} \ast \dot{\mathbb{Q}})$ is the quotient of two forcing notions of precalibre-$\delta$. If $j(\lambda) > \delta$, then since the forcing notion $j(\mathbb{P})/ (\mathbb{P} \ast \dot{\mathbb{Q}})$ has cardinality $<j(\lambda)$, it automatically has precaliber $j(\lambda)$. Otherwise, we use the following claim:
\begin{lemma}
Let $\mathbb{R}$ and $\mathbb{S}$ be two precaliber $\delta$ forcing notions. Let us assume that for every finite collection of conditions (either in $\mathbb{R}$ or in $\mathbb{S}$), if it has a lower bound then it has an unique greatest lower bound. Let $\pi \colon \mathbb{R} \to \mathbb{S}$ be a projection that respects greatest lower bounds of finite collections. 

Then the quotient forcing has precaliber $\delta$.
\end{lemma}
\begin{proof}
Let $\dot{K}$ be the canonical name for the generic filter for $\mathbb{S}$. Let $\dot{I}$ be a name for a subset of conditions in $\mathbb{R} / \dot{K}$ of size $\delta$ that has no subset of cardinality $\delta$ in which every finitely many conditions are compatible.

Let us pick conditions $s_i \in \mathbb{S}$ such that $s_i \Vdash \check{r}_i \in \dot{I}$, $r_i$ are all distinct and $i  <\delta$. Note that in particular, $s_i \leq \pi(r_i)$. Since $\mathbb{S}$ has precaliber-$\delta$, there is a subset of $\delta$ of cardinality $\delta$, $J$, such that any finitely many conditions $s_{\xi_0}, \dots, s_{\xi_{n-1}}$, $\xi_0, \dots, \xi_{n-1}\in J$, are compatible. Since $\mathbb{R}$ has precaliber-$\delta$, there is a set $J^\prime \subseteq J$ of cardinality $\delta$ such that for every $\xi_0, \dots, \xi_{n-1}\in J^\prime$, $r_{\xi_0}, \dots, r_{\xi_{n-1}}$ are compatible. Note that is $s$ is the greatest lower bound of $s_{\xi_0}, \dots, s_{\xi_{n-1}}$ and $r$ is the greatest lower bound of $r_{\xi_0}, \dots, r_{\xi_{n-1}}$ then $s \leq \pi(r)$ and therefore $s\Vdash \check{r} \in \mathbb{R} / \dot{K}$.

Let us show that there is a condition $s\in \mathbb{S}$ that forces $\{\xi \in J^\prime \mid s_{\xi}\in \dot{K}\}$ is unbounded. Otherwise, by the chain condition of $\mathbb{S}$, there was a bound that was forced by the weakest condition of $\mathbb{S}$, $\beta$. But for $\xi > \beta$, $\xi\in J^\prime$, $s_\xi$ provides the contradiction.
\end{proof}

Let $\mathcal{A}$ be an algebra on $j(\lambda)$. Let us extend $j$ to an embedding $\tilde{j}$ by forcing with $j(\mathbb{P} \ast \dot{\mathbb{Q}})/\big(\mathbb{P} \ast \dot{\mathbb{Q}}\big)$. In $M[j(G\ast H)]$, $\tilde{j}^{\prime\prime} \mathcal{A}$ is an elementary substructure of $j(\mathcal{A})$ of cardinality $j(\lambda)$. We want to show that it is $\MM$-elementary.

Since the forcing $j(\mathbb{P} \ast \dot{\mathbb{Q}})/\big(\mathbb{P} \ast \dot{\mathbb{Q}}\big)$ is $j(\lambda)$-$\MM$ preserving, every $\MM$ formula that holds in $j^{\prime\prime}\mathcal{A}$ (and hence in $\mathcal{A}$) in $M[j(G)][j(H)]$, holds in $\mathcal{A}$ in $V[G][H]$ as well. Therefore, it holds in $j(\mathcal{A})$ - as wanted. \end{proof}

It is interesting to check where exactly the proof fails when using Levy collapse instead of $\mathbb{EC}$. Indeed, the only point in which we use a property of $\mathbb{EC}$ which fails for Levy collapse is the existence of a projection from $\mathbb{EC}(\mu, <\delta)$ to $\mathbb{EC}(\kappa, <\delta)$. While there is such projection, there is no projection from $\Col(\mu, <\delta)$ to $\Col(\kappa, <\delta)$ for $\kappa > \mu$ and $\cf \delta \geq \mu^+$.

Taking $\lambda = \kappa$ and $\mu$ regular we obtain $\mu^{++}\mmchang\mu^+$. The proof shows that the result holds also for languages of cardinality $<\kappa$, so we can write $\mu^{++}\xrightarrowdbl[\MM]{\mu}\mu^+$. Assuming that \GCH holds in the ground model, it also holds in the generic extension.
\begin{corollary}
It is consistent, relative to a huge cardinal, that $\aleph_3\mmchang\aleph_2$ and \GCH holds.
\end{corollary}
Using $\lambda=\kappa^{+\omega+1}$ in order to obtain a gap we can get:
\begin{corollary}
It is consistent, relative to a 2-huge cardinal, that $\aleph_{\omega\cdot 2 + 1}\mmchang \aleph_{\omega+1}$.
\end{corollary}
\subsection{MM reflection to \texorpdfstring{$\aleph_1$}{aleph1}}\label{subsec: mm to aleph1 successor to regular}
In the subsection we will show how to derive instances of $\MM$-reflection from some cardinal to $\aleph_1$ using a sufficiently good ideal. For a survey about ideals and their connection to large cardinals, see \cite{Foreman2010ideals}. In particular we will obtain the consistency of $\aleph_2\mmchang\aleph_1$ from a measurable cardinal.

In an unpublished work from 1978, Shelah showed that the instance of Magidor-Malitz reflection, $\aleph_2\mmchang\aleph_1$, is consistent relative to a Ramsey cardinal. Our proof gives weaker consistency result, but it shows an implication between the existence of sufficiently good generically large cardinal and $\MM$-reflection.

Let $\mathcal{I}$ be an ideal on $\kappa$ such that:
\begin{enumerate}
\item $\{\alpha\}\in \mathcal{I}$ for all $\alpha < \kappa$.
\item $\mathcal{I}$ is $\kappa$-complete.
\item The forcing that adds a generic ultrafilter to $\mathcal{P}(\kappa)/\mathcal{I}$ is $\omega+1$-strategically closed.
\end{enumerate}

\begin{theorem}
Assume that there is ideal $\mathcal{I}$ as above and $\diamondsuit(\omega_1)$. Then $\kappa\mmchang \omega_1$.
\end{theorem}
\begin{proof}
The proof follows closely the proof of the completeness theorem for the logic $\mathcal{L}(\MM)$ - the first order logic extended by the Magidor-Malitz quantifier. This result requires $\diamondsuit(\omega_1)$ as well. See \cite[Section 7.3]{Hodges85}. The proof also resemble the proofs in  \cite{farah2006absoluteness}. In this paper, similar methods are used in order to construct models of size $\aleph_1$ that witness the completeness of extensions of $\mathcal{L}(\MM)$ (some of them under large cardinal assumptions).

We start with a countable model $M_0$, and repeatedly add new elements to it. At each step we essentially enlarge $M_\alpha$ by adding some ordinal $\zeta < \kappa$ which is generic over $M_\alpha$ for the forcing $\mathcal{P}(\kappa)/\mathcal{I}$ (i.e. $\{A\in M_\alpha\mid \zeta \in A\}$ is an $M_\alpha$-generic ultrafilter). Eventually, we will show that we can arrange the limit model $M_{\omega_1}$ to be a $\MM$  elementary submodel of some elementary substructure of $H(\chi)$ ($\chi$ large enough regular) of cardinality $\kappa$ that contains all ordinals below $\kappa$. Note that this is the general case, as for any algebra $\mathcal{A}$ on $\kappa$, we may assume that $\mathcal{A}\in M_0$.

Throughout the rest of the proof, $\chi$ is a regular cardinal above $2^{2^{\kappa}}$.
\begin{lemma}
Let $M$ be a countable elementary substructure of $H(\chi)$. There is $\zeta < \kappa$ such that $\{A\in M\mid \zeta\in A\}$ is $M$-generic for the forcing $(\mathcal{P}(\kappa)/\mathcal{I})^M$. Moreover, $M^{\star} := \{f(\zeta) \mid f\colon \kappa \to V,\,f\in M\}$ is a proper extension of $M$ and $M^\star\prec H(\chi)$.
\end{lemma}
\begin{proof}
Let $\{I_n \mid n < \omega\}$ list all the maximal antichains of the forcing $(\mathcal{P}(\kappa)/\mathcal{I})^M$ in $M$. Since the forcing, consists of the $\mathcal{I}$-positive sets is $\sigma$-strategically closed, there is a sequence of $\mathcal{I}$-positive sets $\langle A_n \mid n < \omega\rangle$ such that $A_{n+1}\subseteq A_n$, $A_n\in I_n$  and $\bigcap_{n<\omega} A_n\notin \mathcal{I}$. Any $\zeta\in \bigcap A_n$ will generate a $M$-generic filter.

Since for every $x\in M$ the constant function $c_x(\alpha) = x$ is in $M$, $M\subseteq M^\star$. Since the identity function $id(\alpha) = \alpha$ is in $M$, $\zeta\in M^\star$ so $M^\star$ is strictly larger than $M$.

In order to show that $M^\star\prec H(\chi)$ we use Tarski–Vaught criterion. Let $\varphi(x, b)$ be a formula with $b \in M^\star$, and assume $H(\chi)\models \exists x \varphi(x, b)$. We need to show that $M^\star \models \exists x \varphi(x, b)$. $b = g(\zeta)$ for some $g\in M$. Let \[B = \{\alpha < \kappa \mid H(\chi)\models \exists x \varphi(x, g(\alpha))\}.\]
$B$ is definable from parameters in $M$ and therefore it is a member of $M$. $B\notin \mathcal{I}$, since otherwise, we would have that $\zeta\notin B$. Thus, applying the axiom of choice inside of $M$, there is a function $f\in M$ that assign to every element $\alpha \in B$ a witness $f(\alpha)$ such that $\varphi(f(\alpha), g(\alpha))$. In $M$,
\[B = \{\alpha < \kappa \mid \varphi(f(\alpha), g(\alpha))\}\]
and by elementarity, the same holds in $V$. Since $\zeta\in B$, $\varphi(f(\zeta), g(\zeta))$, so $f(\zeta)$ witnesses $M^\star\models \exists x \varphi(x, b)$.
\end{proof}

Let us define a sequence of models. $M_0 = M$, $M_{\alpha + 1} = M_{\alpha}^\star$ (we will define the $M_\alpha$-generic filters more explicitly in the course of the proof). For limit ordinal $\beta \leq \omega_1$, $M_\beta = \bigcup_{\alpha<\beta} M_\alpha$.

We would like to get that $M_{\omega_1}$ witnesses an instance of $\kappa\mmchang\omega_1$. In order to achieve this, during the iteration we will pick the generic elements in a way that will handle any potential counterexample for the reflection $M_{\omega_1} \cap \kappa \prec_{\MM} H(\chi)\cap \kappa$.

Let $\phi(x)$ be a formula (with parameters from $M$) and let $A$ be an $\mathcal{I}$-positive set. we define the following formula in the language of set theory:
\[\partial_A \phi(w) := ``w\text{ is a function from }\kappa\text{ and } \{\alpha\in A \mid \neg\phi(w(\alpha))\} \in \mathcal{I}''.\]

For a type $\Phi(x)$ with parameters in $M$, (not necessary in $M$), we define \[\partial_A \Phi(w) = \{\partial_A \phi \mid \phi(x)\in\Phi(x)\}.\]

These types control which types will be omitted in the next step of the construction. If $w$ realizes $\partial_A\Phi$ in $M$, then for every choice of $\zeta\in A$, except an $\mathcal{I}$-null set, $\Phi$ is realized in $M^\star$. On the other hand, if for every $\zeta\in A$, $\Phi$ is realized in $M^\star$ (where it is defined using $\zeta$) then there is some positive set $B\subseteq A$ such that $\partial_B\Phi$ is realized. Otherwise, we could remove, outside of $M$, an $\mathcal{I}$-null set and verify that this is not the case. This process cannot be done inside of $M$, since in general $\Phi\notin M$.

One can repeat this process countably many times (using the strategic closure of the forcing) and verify that for a countable set of types $\{\Phi_n(x)\mid n < \omega\}$ if $\partial_A\Phi_n$ is omitted in $M$ for all $n<\omega$ and $A\in M$ then $\Phi_n(x)$ is omitted in $M^\star$.

In $M$, there are names for a positive sets in $M^\star$. Those are essentially the functions $f\colon \kappa \to \mathcal{I}^+$ that appear in $M$. One can define, for a given formula $\varphi$, a positive set $A$ and a name of a positive set $\dot B$ the formula $\partial_A \partial_{\dot B} \varphi$, in the natural way:

$\partial_A \partial_{\dot B} \varphi(w)$ := ``$w$ is a function with domain $\kappa^2$ and $\{\alpha \in A \mid \neg \partial_{\dot{B}(\alpha)}(w_\alpha)\} \in \mathcal{I}$'', where $w_\alpha(x) = w(\alpha, x)$.

We can continue this way and define the derivative of a type relative to any finite sequence of names of positive sets in the iterated forcing (in the narrow sense: the $m$-th set $\dot B$ is a function from $\kappa^m$ to the positive sets).

Let us enrich the language of set theory by all the members of $M$ (as constants). For simplicity of notations, we will use the fact that $M$ is closed under pairs and we will not distinguish between formulas which are provably equivalent.
\begin{lemma}\label{lem: type of maximal block}
Let $\varphi$ be a formula with $k$ free variables. Let $Z\subseteq M$ be a maximal $\varphi$-cube. Let $\Phi$ be the type \[\{\psi(x, p) \mid p \in M, \forall a\in Z,\,\psi(a, p)\} \bigcup \{x \neq a\mid a\in Z\}.\] If $V \models \neg Q^k \varphi$ then $M$ omits all the derivatives of $\Phi$.
\end{lemma}
\begin{proof}
$\Phi$ contains the formulas $\varphi(a_0, \dots, a_{k-2}, x)$ for all $a_i\in Z$ and therefore $M$ does not realize $\Phi$, by the maximality of $Z$.

Let us denote by $\forall^\star \alpha \varphi(\alpha)$ the assertion that $\{\alpha \mid \neg \varphi(\alpha)\}\in \mathcal{I}$.

Assume that $M$ realizes $\partial_{A_0} \partial_{\dot A_1} \cdots \partial_{\dot A_{m-1}} \Phi$ for some $A_0, \dots, \dot{A}_{m-1}\in M$. So there is some $b\in M$ such that:
\[\forall^\star \alpha_0\in A_0 \forall^\star \alpha_1\in \dot{A}_1(\alpha_0) \cdots \forall^\star\alpha_{m-1}\in \dot{A}_{m-1}(\alpha_0, \dots, \alpha_{m-2}) \psi(b(\alpha_0, \dots, \alpha_{m-1}))\]
for every $\psi\in\Phi$.

We may assume that for all $x\in M$, $\forall^\star \alpha_0, \dots \forall^\star \alpha_{m-1}\ b(\alpha_0, \dots, \alpha_{m-1}) \neq x$. For all relevant ordinal (one which escape all the $\mathcal{I}$-null sets in the quantifiers), this is true by the maximality of $Z$ and the fact that $b$ is "forced" to be different than all members of $Z$ in $M$. We can complete the rest of the values (which are essentially elements outside the sets in $\range A_i$, $i< m$ and $A_0$) with dummy values.

Taking $\psi(x)$ to be $\varphi(a_0, \dots, a_{k-2}, x)$ (and omitting the evaluations in the $\dot A_i$) we get:
\[\forall^\star \alpha_0\in A_0 \forall^\star \alpha_1\in \dot A_1 \cdots \forall^\star\alpha_{m-1}\in \dot A_{m-1} \varphi(a_0, \dots, a_{k-2}, b(\alpha_0, \dots, \alpha_{m-1}))\]
By the definition of $\Phi$, replacing $a_{k-2}$ by the variable $x$, we obtain a formula in $\Phi$.
So we conclude that:
\[
\begin{matrix}\forall^\star \alpha_0\in A_0 \forall^\star \alpha_1\in \dot A_1 \cdots \forall^\star\alpha_{m-1}\in \dot A_{m-1} \\
\forall^\star \beta_0\in A_0 \forall^\star \beta_1\in \dot A_1 \cdots \forall^\star\beta_{m-1}\in \dot A_{m-1} \\
\varphi(a_0, \dots, a_{k-3}, b(\beta_0, \dots, \beta_{m-1}), b(\alpha_0, \dots, \alpha_{m-1}))
\end{matrix}\]
Repeating this process and relabeling:
\[\begin{matrix}
\forall^\star \alpha^0_0\in A_0 \forall^\star \alpha^0_1\in \dot A_1 \cdots \forall^\star\alpha^0_{m-1}\in \dot A_{m-1} \\
\forall^\star \alpha^1_0\in A_0 \forall^\star \alpha^1_1\in \dot A_1 \cdots \forall^\star\alpha^1_{m-1}\in \dot A_{m-1} \\
\vdots \\
\forall^\star \alpha^{k-1}_0\in A_0 \forall^\star \alpha^{k-1}_1\in \dot A_1 \cdots \forall^\star\alpha^{k-1}_{m-1}\in \dot A_{m-1} \\
\varphi(b(\alpha^0_0, \dots, \alpha^0_{m-1}),
b(\alpha^1_0, \dots, \alpha^1_{m-1}),\dots,
b(\alpha^{k-1}_0,\dots, \alpha^{k-1}_{m-1}))
\end{matrix}\]
Let us look on this last formula (which is true in $M$) and let us say that a set $D$ is \emph{solid} iff for all $a_0, \dots, a_{r-1} \in D$ ($0 \leq r \leq k$),
\[\begin{matrix}
\forall^\star \alpha^0_0\in A_0 \forall^\star \alpha^0_1\in \dot A_1 \cdots \forall^\star\alpha^0_{m-1}\in \dot A_{m-1} \\
\forall^\star \alpha^1_0\in A_0 \forall^\star \alpha^1_1\in \dot A_1 \cdots \forall^\star\alpha^1_{m-1}\in \dot A_{m-1} \\
\vdots \\
\forall^\star \alpha^{r - k - 1}_0\in A_0 \forall^\star \alpha^{r - k - 1}_1\in \dot A_1 \cdots \forall^\star\alpha^{r - k - 1}_{m-1}\in \dot A_{m-1} \\
\varphi(a_0, \dots, a_{r-1},
b(\alpha^0_0, \dots, \alpha^0_{m-1}),
b(\alpha^1_0, \dots, \alpha^1_{m-1}),
\dots,
b(\alpha^{k-r-1}_0,\dots, \alpha^{k - r - 1}_{m-1}))
\end{matrix}\]
The empty set is solid. Using Zorn's lemma in $M$, we can find a maximal solid set, $D\in M$.
\begin{lemma}
$M\models |D| = \kappa$.
\end{lemma}
\begin{proof}
Assume otherwise. We will find $c\in M$ and outside $D$ such that $\{c\}\cup D$ is solid. If $b^{-1}(D)$ is $\mathcal{I}$-positive then there must be some $d\in D$ such that $b^{-1}(\{d\})$ is $\mathcal{I}$-positive, and we assumed that this is not the case.

Let us iteratively narrow down, in $V$, the positive sets $A_0,\dot{A}_1,\dots$ and replace the quantifier $\forall^\star$ by $\forall$. We would still remain with positive sets. Moreover, we may assume that all of them are disjoint from $b^{-1}(D)$. Pick any $\alpha_0\in A_0$, $\alpha_1\in \dot{A}_1(\alpha_0), \dots, \alpha_{m-1}\in \dot A_{m-1}(\alpha_0, \dots, \alpha_{m-2})$. Let $c = b(\alpha_0, \dots, \alpha_{m-1})$. $c\notin D$ and for every $a_i\in D$, $\varphi(a_0, \dots, a_{k-2}, c)$, as wanted.
\end{proof}

We conclude that $D$ is a $\varphi$-cube of cardinality $\kappa$. But by elementarity, $D$ is a $\varphi$-cube in $V$ as well.
\end{proof}
In order to finish the proof we need to explain how to choose the sets $Z$. Here the diamond comes into the picture. Let $\langle S_\alpha \mid \alpha < \omega_1\rangle$ a $\diamondsuit(\omega_1)$ sequence. For convenience, we will assume that $S_\alpha$ is a pair $(A_\alpha, \phi_\alpha)$ where $A_\alpha\subseteq \alpha$ and $\phi_\alpha$ is a $\MM$-formula $Q^n \varphi$ with parameters in $\alpha$.

For every $i < \omega_1$, let us choose a bijection between $M_{i+1}\setminus M_i$ and $\omega \cdot (i+1) \setminus \omega \cdot i$. Connecting those bijections we obtain a continuous bijection between $M_{\omega_1}$ and $\omega_1$.

For every $\alpha$, if $A_\alpha$ is a maximal $\varphi_\alpha$-cube in $M_\alpha$, we define $\Phi_\alpha$ to be the type which was defined in lemma \ref{lem: type of maximal block}. Otherwise, we do nothing.

When enlarging $M_i$ to be $M_{i+1}$ we omit the types $\{\Phi_j \mid j \leq i\}$. Let $\psi = Q^n \varphi$ be a $\MM$-formula. Assume that $M_{\omega_1}$ satisfies $\psi$ and that $Z$ is a maximal $\varphi$-cube. Then on club many points, $Z\cap \alpha$ is a maximal $\varphi$-cube. Therefore, there is a point $\alpha < \omega_1$ such that $Z\cap \alpha = A_\alpha$, $\varphi_\alpha = \varphi$. But the corresponding type was not omitted, since it was enlarged, so $V\models\psi$, as needed.
\end{proof}
We remark that for successor of a regular cardinal $\kappa$, the existence of such an ideal $\mathcal{I}$ is equiconsistent with the existence of a measurable cardinal. Unfortunately, for successor of singular cardinals of countable cofinality, such ideal cannot exist.
\begin{question}\label{ques: mm reflection at aleph omega}
Is $\aleph_{\omega+1}\xrightarrow[\MM]{}\aleph_1$ consistent?
\end{question}
\section{Acknowledgements}
I would like to thank my advisor, Menachem Magidor, for helping me formulate the questions on which this paper is based, as well as suggesting directions for solving them. I would like to thank Shimon Garti for reading a previous version of the paper and suggesting many improvements for the presentation of the ideas of the paper. I would like to thank Monroe Eskew for pointing out an error in the early version of the paper.  

Lastly, I would like to thank the anonymous referee who helped to make this paper more correct and readable.

\providecommand{\bysame}{\leavevmode\hbox to3em{\hrulefill}\thinspace}
\providecommand{\MR}{\relax\ifhmode\unskip\space\fi MR }
\providecommand{\MRhref}[2]{%
  \href{http://www.ams.org/mathscinet-getitem?mr=#1}{#2}
}
\providecommand{\href}[2]{#2}


\begin{thebibliography}{10}

\bibitem{ChangKeisler1990}
C.~C. Chang and H.~J. Keisler, \emph{Model theory}, third ed., Studies in Logic
  and the Foundations of Mathematics, vol.~73, North-Holland Publishing Co.,
  Amsterdam, 1990. \MR{1059055 (91c:03026)}

\bibitem{eskew2014measurability}
Monroe~Blake Eskew, \emph{Measurability properties on small cardinals
  dissertation}, Ph.D. thesis, UNIVERSITY OF CALIFORNIA, IRVINE, 2014.

\bibitem{farah2006absoluteness}
Ilijas Farah, Richard Ketchersid, Paul Larson, and Menachem Magidor,
  \emph{{A}bsoluteness for universally baire sets and the uncountable ii},
  Computational prospects of infinity. Part II. Presented talks \textbf{15}
  (2006), 163--191.

\bibitem{Foreman2010ideals}
Matthew Foreman, \emph{Ideals and generic elementary embeddings}, Handbook of
  set theory, Springer, 2010, pp.~885--1147.

\bibitem{ForemanLaver1988}
Matthew Foreman and Richard Laver, \emph{Some downwards transfer properties for
  {$\aleph_2$}}, Adv. in Math. \textbf{67} (1988), no.~2, 230--238. \MR{925267
  (89h:03090)}

\bibitem{Hodges85}
Wilfrid Hodges, \emph{Building models by games}, London Mathematical Society
  Student Texts, vol.~2, Cambridge University Press, Cambridge, 1985.
  \MR{812274 (87h:03045)}

\bibitem{Kunen1978}
Kenneth Kunen, \emph{Saturated ideals}, J. Symbolic Logic \textbf{43} (1978),
  no.~1, 65--76. \MR{495118 (80a:03068)}

\bibitem{Hanson2015}
Chris Lambie-Hanson, \emph{Squares and narrow systems}, arXiv preprint
  arXiv:1510.04067 (2015).

\bibitem{LevinskiMagidorShelah}
Levinski, Jean-Pierre, Menachem Magidor, and Saharon Shelah. \emph{Chang's conjecture for {$\aleph_{\omega}$}.} Israel J. Math. \textbf{69.2} (1990), 161--172.

\bibitem{Magidor1971}
Menachem Magidor, \emph{On the role of supercompact and extendible cardinals in
  logic}, Israel J. Math. \textbf{10} (1971), 147--157. \MR{0295904 (45
  \#4966)}

\bibitem{MagidorMalitz1977}
Menachem Magidor and Jerome Malitz, \emph{Compact extensions of {$L(Q)$}.
  {I}a}, Ann. Math. Logic \textbf{11} (1977), no.~2, 217--261. \MR{0453484 (56
  \#11746)}

\bibitem{MagidorShelah96}
Menachem Magidor and Saharon Shelah, \emph{The tree property at successors of
  singular cardinals}, Archive for Mathematical Logic \textbf{35} (1996),
  no.~5-6, 385--404.

\bibitem{NeemanSteelSubcompact}
Itay Neeman and John Steel, \emph{Equiconsistencies at subcompact cardinals},
  submitted.

\bibitem{Rinot2014}
Assaf Rinot, \emph{Chain conditions of products, and weakly compact cardinals},
  Bull. Symb. Log. \textbf{20} (2014), no.~3, 293--314. \MR{3271280}

\bibitem{ShelahRubin}
Matatyahu Rubin and Saharon Shelah, \emph{On the expressibility hierarchy of
  {M}agidor-{M}alitz quantifiers}, J. Symbolic Logic \textbf{48} (1983), no.~3,
  542--557. \MR{716614}

\bibitem{SharonViale2010}
Assaf Sharon and Matteo Viale, \emph{Some consequences of reflection on the
  approachability ideal}, Transactions of the American Mathematical Society
  \textbf{362} (2010), no.~8, 4201--4212.

\bibitem{Shelah-MM-compactness-lambda}
Saharon Shelah, \emph{Models with second order properties. {III}. {O}mitting
  types for {$L(Q)$}}, Arch. Math. Logik Grundlag. \textbf{21} (1981), no.~1-2,
  1--11. \MR{625527}

\bibitem{Shioya2011}
Masahiro Shioya, \emph{The {E}aston collapse and a saturated filter},  (2011).

\end{thebibliography}
\end{document}